\documentclass[a4paper,10pt,reqno]{amsart}
\newtheorem{theorem}{Theorem}[section]

\newtheorem{maintheorem}{Theorem}
\newtheorem{secondtheorem}[maintheorem]{Theorem}

\newtheorem{proposition}[theorem]{Proposition}

\newtheorem{lemma}[theorem]{Lemma}
\theoremstyle{definition}
\newtheorem{definition}[theorem]{Definition}

\newtheorem{remark}[theorem]{Remark}

\newcommand{\sing}{\textsf{Sing}}
\newcommand{\dm}{\textsf{dim}}

\numberwithin{equation}{section}

\newcommand{\Diffeo}{\operatorname{Diff}}

\begin{document}
\title[Levi-flat hypersurfaces with an isolated line singularity]{On Normal forms for Levi-flat hypersurfaces with an isolated line singularity}
\date{\today}
\author{Arturo Fern\' andez P\' erez}
\address{Departamento de Matem\'atica, Universidade Federal de Minas Gerais, UFMG}
\curraddr{Av. Ant\^onio Carlos, 6627 C.P. 702, 30123-970 - Belo Horizonte - MG, Brazil.}
\email{arturofp@mat.ufmg.br}
\begin{abstract}
We prove the existence of normal forms for some local real-analytic Levi-flat hypersurfaces with an isolated line singularity. We also give sufficient conditions for that a Levi-flat hypersurface with a complex line as singularity to be a pullback of a real-analytic curve in $\mathbb{C}$ via a holomorphic function.
\end{abstract}
\maketitle
\section{Introduction}
\par Let $M\subset U\subset\mathbb{C}^{n}$ be a real-analytic hypersurface, where $U$ is an open set and denote by  $M^{*}$ the regular part, that is, near each point $p\in M^{*}$, the variety $M$ is a manifold of real codimension one. For each  $p\in M^{*}$, there is a unique complex hyperplane $L_{p}$ contained in the tangent space $T_{p}M^{*}$, and consequently defines a real-analytic distribution $p\mapsto L_{p}$ of complex hyperplanes in $T_{p}M^{*}$, the so-called \textit{Levi distribution}. We say that $M$ is \textit{Levi-flat}, if the Levi distribution is integrable in sense of Frobenius.  The foliation defined by this distribution is called \textit{Levi-foliation}. The local structure near regular points is very well  understood, according to E. Cartan, around each $p\in M^{*}$ we can find local holomorphic coordinates $z_{1},\ldots,z_{n}$ such that $M^{*}=\{\mathcal{R}e(z_{n})=0\}$, and consequently the leaves of Levi-foliation are 
imaginary levels of $z_{n}$. The singular case was studied by Burns-Gong \cite{burns}, The authors classified singular Levi-flat hypersurfaces in $\mathbb{C}^{n}$ with quadratic singularities and also 
proved the existence of a normal form, in the case of generic (Morse) singularities. In \cite{alcides}, Cerveau-Lins Neto have proved that a local real-analytic Levi-flat hypersurface $M$ with a sufficiently small singular set is 
given by the zeros of the real part of a holomorphic function. 
\par The aim of this paper is to prove
 the existence of some normal forms for local real-analytic Levi-flat hypersurfaces defined by the vanishing of real part of holomorphic functions with an \textit{isolated line singularity} (for short: ILS). In particular, we establish an analogous result like in Singularity Theory for germs of holomorphic functions. 
  \par The main motivation for this work is a result due to Dirk Siersma, who introduced in \cite{siersma} the class of germs of holomorphic functions with an ILS. More precisely, let $\mathcal{O}_{n+1}:=\{f:(\mathbb{C}^{n+1},0)\rightarrow\mathbb{C}\}$ be the ring of germs of holomorphic functions and let $m$ be its maximal ideal. If $(x,y)=(x,y_{1},\ldots,y_{n})$ denote the coordinates in $\mathbb{C}^{n+1}$ and consider the line $L:=\{y_1=\ldots=y_n=0\}$, let $I:=(y_{1},\ldots,y_n)\subset\mathcal{O}_{n+1}$ be its ideal and denote by $\mathcal{D}_{I}$ the group of local analytic isomorphisms $\varphi:(\mathbb{C}^{n+1},0)\rightarrow(\mathbb{C}^{n+1},0)$ for which $\varphi(L)=L$. Then $\mathcal{D}_{I}$ acts on $I^2$ and for $f\in I^2$, the tangent space of (the orbit of) $f$ with respect to this action is the ideal defined by $$\tau(f):=m.\frac{\partial{f}}{\partial{x}}+I.\frac{\partial{f}}{\partial{y}}$$
and the codimension of (the orbit) of $f$ is $$c(f):=\dim_{\mathbb{C}}\frac{I^2}{\tau(f)}.$$
\par A line singularity is a germ $f\in I^2$. An ILS is a line singularity $f$ such that $c(f)<\infty$. Geometrically, $f\in I^2$ is an ILS if and only if the singular locus of $f$ is $L$ and for every $x\neq 0$, the germ of (a representative of) $f$ at $(x,0)\in L$ is equivalent to $y^{2}_{1}+\ldots+y^2_{n}$. In a certain sense ILS are the first generalization of isolated singularities. D. Siersma proved the following result. (The topology on $\mathcal{O}_{n+1}$ is introduced as in \cite[p. 145]{durfe}).
\begin{theorem}\label{teo_uno}
A germ $f\in I^2$ is $D_{I}$-simple (i.e. $c(f)<\infty$ and $f$ has a neighborhood in $I^2$ which intersects only a finite number of $D_{I}$-orbits) if and only if $f$ is $D_{I}$-equivalent to one the germs in the following table

$$\begin{tabular}{|c|l|r|}\hline
Type & Normal form & Conditions \\
\hline
$A_{\infty}$ & $y_{1}^{2}+y_{2}^{2}+\ldots +y^{2}_{n}$ & \\
\hline
$D_{\infty}$ & $xy^{2}_{1}+y^{2}_{2}+\ldots+y^{2}_{n}$ & \\
\hline
$J_{k,\infty}$ & $x^{k}y^{2}_{1}+y_{1}^{3}+y^{2}_{2}+\ldots+y_{n}^{2}$ & $k\geq 2$ \\
\hline
$T_{\infty,k,2}$ & $x^{2}y^{2}_{1}+y_{1}^{k}+y_{2}^{2}+\ldots+y_{n}^{2}$ & $k\geq 4$ \\
\hline
$Z_{k,\infty}$ & $xy_{1}^{3}+x^{k+2}y_{1}^{2}+y_{2}^{2}+\ldots+y_{n}^{2}$ & $k\geq 1$ \\
\hline
$W_{1,\infty}$ & $x^{3}y_{1}^{2}+y_{1}^{4}+y_{2}^{2}+\ldots+y_{n}^{2}$ &  \\
\hline
$T_{\infty,q,r}$ & $xy_{1}y_{2}+y_{1}^{q}+y_{2}^{r}+y^{2}_{3}\ldots+y_{n}^{2}$ & $q\geq r\geq 3$ \\
\hline
$Q_{k,\infty}$ & $x^{k}y_{1}^{2}+y_{1}^{3}+xy_{2}^{2}+y^{2}_{3}\ldots+y_{n}^{2}$ & $k\geq 2$ \\
\hline
$S_{1,\infty}$ & $x^{2}y_{1}^{2}+y_{1}^{2}y_{2}+y_{3}^{2}+\ldots+y_{n}^{2}$ &  \\

\hline
\end{tabular}$$
\begin{center}
Table 1. Isolated Line singularities 
\end{center}
\end{theorem}
\par The singularities in Theorem \ref{teo_uno} are analogous of the $A$-$D$-$E$ singularities due to Arnold \cite{singulararnold}. A new characterization of simple ILS have been proved by A. Zaharia \cite{zaharia}. 
We prove the existence of normal forms for Levi-flat hypersurfaces with an ILS.
\begin{maintheorem}\label{theo-uno}
Let $M=\{F=0\}$ be a germ of an irreducible real-analytic hypersurface on $(\mathbb{C}^{n+1},0)$, $n\geq 3$. Suppose that  
\begin{enumerate}
\item $F(x,y)=\mathcal{R}e(P(x,y))+H(x,y),$ where $P(x,y)$ is one of the germs of the Table 1.
 \item $M=\{F=0\}$ is Levi-flat.
  \item $H(x,0)=0$ for all $x\in(\mathbb{C},0)$ and $j_{0}^{k}(H)=0$, for $k=\deg(P)$.
\end{enumerate}
 Then there exists a biholomorphism $\varphi:(\mathbb{C}^{n+1},0)\rightarrow(\mathbb{C}^{n+1},0)$ preserving $L$ such that
$$\varphi(M)=\{\mathcal{R}e(P(x,y))=0\}.$$
\end{maintheorem}
\par This result is a Siersma's type Theorem for singular Levi-flat hypersurfaces. We remark that the function $H$ is of course restricted by the assumption that $M$ is Levi flat. Now, if $\varphi(M)=\{\mathcal{R}e(P(x,y))=0\}$, where $P$ is a germ with an ILS at $L$ then $\sing(M)=L$. In other words, $M$ is a Levi-flat hypersurface with an ILS at $L$. 
If $P(x,y)$ is the germ $A_{\infty}$, 
we prove that Theorem \ref{theo-uno} is true in the case $n=2$. 
\begin{secondtheorem}\label{theo-dos}
Let $M=\{F=0\}$ be a germ of an irreducible real-analytic Levi-flat hypersurface on $(\mathbb{C}^{3},0)$. Suppose that $F$ is defined by $$F(x,y)=\mathcal{R}e(y_{1}^{2}+y_{2}^{2})+H(x,y),$$
where $H$ is a germ of real-analytic function such that $H(x,0)=0$ and $j_{0}^{k}(H)=0$ for $k=2$.
Then there exists a biholomorphism $\varphi:(\mathbb{C}^{3},0)\rightarrow(\mathbb{C}^{3},0)$ preserving $L$ such that
$\varphi(M)=\{\mathcal{R}e(y_{1}^{2}+y_{2}^{2})=0\}$.
\end{secondtheorem}
\par The above result should be compared to \cite[Theorem 1.1]{burns}. This result can be viewed as a Morse's Lemma for Levi-flat hypersurfaces with an ILS at $L$. The problem of normal forms of Levi-flat hypersurfaces in $\mathbb{C}^{3}$ with an ILS seems difficult in the other cases.  To prove these results we use techniques of holomorphic foliations developed in \cite{alcides} and \cite{tese}. Another normal forms of singular Levi-flat hypersurfaces have been obtained in \cite{burns}, \cite{normal} and \cite{singular}.
\par This paper is organized as follows: In Section 2, we recall some definitions and known results about Levi-flat and holomorphic foliations. Section 3 is devoted to prove Theorem \ref{theo-uno}. In Section 4, we prove Theorem \ref{theo-dos}. Finally, in Section 5, using holomorphic foliations, we give sufficient conditions for that a Levi-flat hypersurface with a complex line as singularity to be a pullback of a real-analytic curve in $\mathbb{C}$ via a holomorphic function, (see Theorem \ref{line_complex}).

\section{Levi-flat hypersurfaces and Foliations}
\par In this section we works with germs at $0\in\mathbb{C}^{n+1}$ of irreducible real-analytic hypersurfaces and of 
codimension one holomorphic foliations. Let $M=\{F=0\}$, where $F:(\mathbb{C}^{n+1},0)\rightarrow(\mathbb{R},0)$ is a germ of an irreducible real-analytic function, and 
$M^{*}:=\{F=0\}\backslash\{dF=0\}$. Let us define the singular set of $M$ (or ``set of critical points" of $M$) by
\begin{equation}\label{singular}
\sing(M):=\{F=0\}\cap\{dF=0\}.
\end{equation} 
Note that $\sing(M)$ contains all points $q\in M$ such that $M$ is smooth at $q$, but the codimension of $M$ at $q$ is at least two. 
In general the singular set of a real-analytic  subvariety $M$ in a complex manifold is defined as the set of points near which $M$ is not a 
real-analytic submanifold (of any dimension) and ``in general" has structure of a semianalytic set; see for instance, \cite{singularlebl}. 
In this paper, we work with $\sing(M)$ as defined in (\ref{singular}). We recall that (in this case)
the Levi distribution $L$ on $M^{*}$ is defined by
\begin{align}
L_{p}:=ker(\partial{F}(p))\subset T_{p}M^{*}=ker(dF(p)),\,\,\,\,\text{for any} \,\,p\in M^{*}.
\end{align}
Let us suppose that $M$ is \textit {Levi-flat}, this implies that $M^{*}$ is foliated by complex codimension one holomorphic submanifolds immersed on
$M^{*}$. 
\par Note that the Levi distribution $L$ on $M^{*}$ can be defined by the real-analytic 1-form $\eta=i(\partial{F}-\bar{\partial}F)$, which is called the Levi 1-form of $F$. It is well known that the integrability condition of $L$ is equivalent to equation $(\partial{F}-\bar{\partial}F)\wedge\partial\bar{\partial}F|_{M^{*}}=0.$ 
\par Let us consider the series Taylor of $F$ at $0\in\mathbb{C}^{n+1}$, $$F(x,y)=\sum_{i,\mu,j,\nu}F_{i \mu j\nu}x^{i}y^{\mu}\bar{x}^{j}\bar{y}^{\nu}$$
where $\bar{F}_{i\mu j\nu}=F_{j\nu i\mu}$; $i,j\in\mathbb{N}$, $\mu=(\mu_{1},\ldots,\mu_{n})$, $\nu=(\nu_{1},\ldots,\nu_{n})$, $(x,y)\in\mathbb{C}\times\mathbb{C}^n$, 
$y^{\mu}=y_{1}^{\mu_{1}}\ldots y_{n}^{\mu_{n}}$ and $\bar{y}^{\nu}=\bar{y}_{1}^{\nu_{1}}\ldots \bar{y}_{n}^{\nu_{n}}$. The complexification $F_{\mathbb{C}}\in\mathcal{O}_{2n+2}$ of $F$ is defined by the serie
$$F_{\mathbb{C}}(x,y,z,w)=\sum_{i,\mu,j,\nu}F_{i\mu j\nu}x^{i}y^{\mu}z^{j}w^{\nu},$$ where $z\in\mathbb{C}$, $w=(w_1,\ldots,w_n)\in\mathbb{C}^{n}$ and $w^{\nu}=w_{1}^{\nu_{1}}\ldots w_{n}^{\nu_{n}}$.
Notice that $F(x,y)=F_{\mathbb{C}}(x,y,\bar{x},\bar{y})$. The complexification $M_{\mathbb{C}}$ of $M$ is defined as $M_{\mathbb{C}}:=\{F_{\mathbb{C}}=0\}$ and defines a complex subvariety in $\mathbb{C}^{2n+2}$, its regular part is $M^{*}_{\mathbb{C}}:=M_{\mathbb{C}}\backslash\{dF_{\mathbb{C}}=0\}$. Now, assume that $M$ is Levi-flat. Then the integrability condition of $$\eta=i(\partial{F}-\bar{\partial}F)|_{M^{*}}$$ implies that $\eta_{\mathbb{C}}|_{M^{*}_{\mathbb{C}}}$ is integrable, where 
 $$\eta_{\mathbb{C}}:=i[(\partial_{x}F_{\mathbb{C}}+\partial_{y}F_{\mathbb{C}})-(\partial_{z}F_{\mathbb{C}}+\partial_{w}F_{\mathbb{C}})].$$
 Therefore
$\eta_{\mathbb{C}}|_{M^{*}_{\mathbb{C}}}$ defines a codimension one holomorphic foliation $\mathcal{L}_{\mathbb{C}}$ on $M^{*}_{\mathbb{C}}$ that will be called the complexification of $\mathcal{L}$.
 
\par Let $W:=M_{\mathbb{C}}^{*}\backslash  \sing(\eta_{\mathbb{C}}|_{M^{*}_{\mathbb{C}}})$ and denote by $L_{\zeta}$ the leaf of $\mathcal{L}_{\mathbb{C}}$ through $\zeta$, where $\zeta\in W$. The next results will be used several times along of the paper.
\begin{lemma}[Cerveau-Lins Neto \cite{alcides}]\label{closed-lemma}
For any $\zeta\in W$, the leaf $L_{\zeta}$ of $\mathcal{L}_{\mathbb{C}}$ through $\zeta$ is closed in $M_{\mathbb{C}}^{*}$.
\end{lemma} 
\begin{definition}
The algebraic dimension of $\sing(M)$ is the complex dimension of the singular set of $M_{\mathbb{C}}$.
\end{definition}
\par The following result will be used enunciated in the context of Levi-flat hypersurfaces in $\mathbb{C}^{n+1}$.
\begin{theorem}[Cerveau-Lins Neto \cite{alcides}]\label{alcides-theorem}
 Let $M=\{F=0\}$ be a germ of an irreducible analytic Levi-flat hypersurface at $0\in\mathbb{C}^{n+1}$, $n\geq{2}$, with Levi 1-form $\eta=i(\partial{F}-\bar{\partial} F)$. Assume that the algebraic dimension of $\sing(M)\leq 2n-2$. Then there exists a unique germ at $0\in\mathbb{C}^{n+1}$ of holomorphic codimension one foliation $\mathcal{F}_{M}$ tangent to $M$, if one of the following conditions is fulfilled:
\begin{enumerate}
\item $n\geq 3$ and $cod_{M_{\mathbb{C}}^{*}}(\sing(\eta_{\mathbb{C}}|_{M_{\mathbb{C}}^{*}}))\geq 3$.
\item  $n\geq 2$, $cod_{M_{\mathbb{C}}^{*}}(\sing(\eta_{\mathbb{C}}|_{M_{\mathbb{C}}^{*}}))\geq 2$
and $\mathcal{L}_{\mathbb{C}}$ admits a non-constant holomorphic first integral.
\end{enumerate}
Moreover, in both cases the foliation $\mathcal{F}_{M}$ admits a non-constant holomorphic first integral $f$ such that $M=\{\mathcal{R}e(f)=0\}$.
\end{theorem}
\section{Proof of Theorem \ref{theo-uno}}
We write
$$F(x,y)=\mathcal{R}e(P(x,y_{1},\ldots,y_{n}))+H(x,y_{1},\ldots,y_{n}),$$
where $P(x,y_{1},\ldots,y_{n})$ is one of the polynomials of the Table 1, $H:(\mathbb{C}^{n+1},0)\rightarrow(\mathbb{R},0)$ is a germ of real-analytic function such that $H(x,0)=0$ for all $x\in(\mathbb{C},0)$ and $j_{0}^{k}(H)=0$, for $k=\deg(P)$. The complexification of $F$ is given by
$$F_{\mathbb{C}}(x,y,z,w)=\frac{1}{2}P(x,y)+\frac{1}{2}P(z,w)+H_{\mathbb{C}}(x,y,z,w),$$
  thus $M_{\mathbb{C}}=\{F_{\mathbb{C}}(x,y,z,w)=0\}\subset(\mathbb{C}^{2n+2},0)$, where $z\in\mathbb{C}$ and $w=(w_1,\ldots,w_n)\in\mathbb{C}^n$.
\par Since $P(x,y)$ has an ILS at $L$, we get $\sing(M_{\mathbb{C}})=\{y=w=0\}\simeq\mathbb{C}^2$. In particular, the algebraic dimension of $\sing(M)$ is $2$. On the other hand, the complexification of $\eta=i(\partial{F}-\bar{\partial}F)$ is $$\eta_{\mathbb{C}}:=i[(\partial_{x}F_{\mathbb{C}}+\partial_{y}F_{\mathbb{C}})-(\partial_{z}F_{\mathbb{C}}+\partial_{w}F_{\mathbb{C}})].$$
Recall that $\eta|_{M^{*}}$ and $\eta_{\mathbb{C}}|_{M^{*}_{\mathbb{C}}}$ define $\mathcal{L}$ and $\mathcal{L}_{\mathbb{C}}$ respectively. Now we compute $\sing(\eta_{\mathbb{C}}|_{M^{*}_{\mathbb{C}}})$.
 We can write $dF_{\mathbb{C}}=\alpha+\beta$, with $$\alpha:=\frac{\partial{F_{\mathbb{C}}}}{\partial{x}}dx+\sum_{j=1}^{n}\frac{\partial{F_{\mathbb{C}}}}{\partial{y}_{j}}dy_{j}=\frac{1}{2}\frac{\partial{P}}{\partial{x}}(x,y)dx+
\frac{1}{2}\sum_{j=1}^{n}\frac{\partial{P}}{\partial{y}_{j}}(x,y)dy_{j}+\theta_{1}$$
and $$\beta:=\frac{\partial{F_{\mathbb{C}}}}{\partial{z}}dz+\sum_{j=1}^{n}\frac{\partial{F_{\mathbb{C}}}}{\partial{w}_{j}}dw_{j}=\frac{1}{2}\frac{\partial{P}}{\partial{z}}(z,w)dz+
\frac{1}{2}\sum_{j=1}^{n}\frac{\partial{P}}{\partial{w}_{j}}(z,w)dw_{j}+\theta_2$$
where $\theta_1=\frac{\partial{H_{\mathbb{C}}}}{\partial{x}}dx+\sum_{j=1}^{n}\frac{\partial{H_{\mathbb{C}}}}{\partial{z}_{j}}dz_{j}$ and
$\theta_2=\frac{\partial{H_{\mathbb{C}}}}{\partial{z}}dz+\sum_{j=1}^{n}\frac{\partial{H_{\mathbb{C}}}}{\partial{w}_{j}}dw_{j}$. 
\par Note that $\eta_{\mathbb{C}}=i(\alpha-\beta)$, and so
\begin{eqnarray}
\eta_{\mathbb{C}}|_{M^{*}_{\mathbb{C}}}=(\eta_{\mathbb{C}}+idF_{\mathbb{C}})|_{M^{*}_{\mathbb{C}}}=
2i\alpha|_{M^{*}_{\mathbb{C}}}=-2i\beta|_{M^{*}_{\mathbb{C}}}.
\end{eqnarray}
In particular, $\alpha|_{M^{*}_{\mathbb{C}}}$ and $\beta|_{M^{*}_{\mathbb{C}}}$
define $\mathcal{L}_{\mathbb{C}}$. Therefore  $\sing(\eta_{\mathbb{C}}|_{M^{*}_{\mathbb{C}}})$ can be split in two parts. In fact, let $M_{1}:=\{(x,y,z,w)\in M_{\mathbb{C}}| \frac{\partial{F}_{\mathbb{C}}}{\partial{z}}\neq 0$ or $\frac{\partial{F}_{\mathbb{C}}}{\partial{w_{j}}}\neq 0$ for some $j=1,\ldots,n\}$ and 
$M_{2}:=\{(x,y,z,w)\in M_{\mathbb{C}}| \frac{\partial{F}_{\mathbb{C}}}{\partial{x}}\neq 0$ or $\frac{\partial{F}_{\mathbb{C}}}{\partial{z_{j}}}\neq 0$ for some $j=1,\ldots,n\}$, then $M_{\mathbb{C}}=M_{1}\cup M_{2}$. If we denote by
$A_{0}=\frac{\partial{H_{\mathbb{C}}}}{\partial{x}}$, $A_{j}=\frac{\partial{H_{\mathbb{C}}}}{\partial{z}_{j}}$ for all $1\leq j\leq n$ and by $B_0=\frac{\partial{H_{\mathbb{C}}}}{\partial{z}}$, $B_j=\frac{\partial{H_{\mathbb{C}}}}{\partial{w}_{j}}$ for all $1\leq j\leq n$, we obtain that  $\sing(\eta_{\mathbb{C}}|_{M^{*}_{\mathbb{C}}})=X_{1}\cup X_{2}$, where
$$X_{1}:=M_{1}\cap\{\frac{\partial{P}}{\partial{x}}(x,y)+A_{0}=\frac{\partial{P}}{\partial{y}_{1}}(x,y)+A_{1}=
\ldots=\frac{\partial{P}}{\partial{y}_{n}}(x,y)+A_{n}=0\}$$ and 
$$X_{2}:=M_{2}\cap\{\frac{\partial{P}}{\partial{z}}(z,w)+B_{0}=\frac{\partial{P}}{\partial{w}_{1}}(z,w)+B_{1}=
\ldots=\frac{\partial{P}}{\partial{w}_{n}}(z,w)+B_{n}=0\}.$$
Since $P$ is a polynomial with an ILS at $L=\{y=0\}$, we  conclude that $$cod_{M^{*}_{\mathbb{C}}}\sing(\eta_{\mathbb{C}}|_{M^{*}_{\mathbb{C}}})=n.$$
\par By hypothesis $n\geq 3$,  then it follows from Theorem \ref{alcides-theorem}, part $(1)$ that there exists a germ $f\in\mathcal{O}_{n+1}$ such that the holomorphic foliation $\mathcal{F}$ defined by $df=0$ is tangent to $M$. Moreover $M=\{\mathcal{R}e(f)=0\}$. Note that if $M=\{\mathcal{R}e(f)=0\}=\{F=0\}$, with $F$ an irreducible germ, we must have that $\mathcal{R}e(f)=U\cdot F$, where $U$ is a germ of real-analytic function with $U(0)\neq 0$. Without loss of generality, we can assume that $U(0)=1$.
In particular, $\mathcal{R}e(f)=U\cdot F$ implies that $f=P+h.o.t$. According to Theorem \ref{teo_uno},  there exists a biholomorphism $\varphi:(\mathbb{C}^{n+1},0)\rightarrow(\mathbb{C}^{n+1},0)$ preserving $L$ such that $f\circ\varphi^{-1}=P$, ($f$ is $D_{I}$-equivalent to $P$, because $f$ is a germ with ILS at $L$). Therefore, $\varphi(M)=\{\mathcal{R}e(P)=0\}$ and the proof ends.

\section{Proof of Theorem \ref{theo-dos}}
The idea is to use Theorem \ref{alcides-theorem}, part (2). In order to prove our result in the case $n=2$, 
 we are going to prove that $\mathcal{L}_{\mathbb{C}}$ has a non-constant holomorphic first integral. 
\par We begin by a blow-up along $C:=\{y_{1}=y_{2}=w_{1}=w_{2}=0\}\simeq\mathbb{C}^{2}\subset\mathbb{C}^6$. Let $F(x,y_{1},y_{2})=\mathcal{R}e(y^{2}_1+y^{2}_2)+H$ and $M=\{F=0\}$ Levi-flat. Its complexification can be written as $$F_{\mathbb{C}}(x,y_1,y_2,z,w_1,w_2)=\frac{1}{2}(y^{2}_1+y^{2}_2)+\frac{1}{2}(w^{2}_1+w^{2}_2)+H_{\mathbb{C}}(x,y_1,y_2,z,w_1,w_2).$$ 
Note that $$\sing(M_{\mathbb{C}})=\{y=w=0\}=C.$$
\par Let $E$ be the exceptional divisor of the blow-up $\pi:\tilde{\mathbb{C}}^{6}\rightarrow\mathbb{C}^{6}$ along $C$. Denote by $\tilde{M}_{\mathbb{C}}:=\overline{\pi^{-1}(M_{\mathbb{C}}\setminus\{C\})}\subset\tilde{\mathbb{C}}^{6}$ the strict transform of  $M_{\mathbb{C}}$ via $\pi$ and by $\tilde{\mathcal{F}}:=\pi^{*}(\mathcal{L}_{\mathbb{C}})$ the foliation on $\tilde{M}_{\mathbb{C}}$.
\par Now, we consider an especial situation.
Suppose that $\tilde{M}_{\mathbb{C}}$ is smooth and set $\tilde{C}:=\tilde{M}_{\mathbb{C}}\cap E$. Moreover, assume that  $\tilde{C}$ is invariant by $\tilde{\mathcal{F}}$. Take $S=\tilde{C}\setminus\sing\tilde{\mathcal{F}}$, then $S$ is a smooth leaf of $\tilde{\mathcal{F}}$. Pick $p_{0}\in S$ and a transverse section $\sum$ through $p_{0}$. Let $G\subset \Diffeo(\sum,p_0)$ be the holonomy group of the leaf $S$ of $\tilde{\mathcal{F}}$. Since $\dm(\sum)=1$, we can assume that 
$G\subset\Diffeo(\sum,0)$. We state a fundamental lemma.
\begin{lemma}[Fern\'andez-P\'erez \cite{singular}]\label{integral-theorem}
In the above situation, suppose that the following properties are verified:
\begin{enumerate}
\item  For any $p\in S\backslash\sing{(\tilde{\mathcal{F}})}$ the leaf $L_{p}$ of $\tilde{\mathcal{F}}$ through $p$ is closed in $S$.
\item  $g'(0)$ is a primitive root of unity, for all $g\in G\backslash\{id\}$.
\end{enumerate}
 Then $\mathcal{L}_{\mathbb{C}}$ admits a non-constant holomorphic first integral.
\end{lemma}
\begin{proof}
\par Let $G'=\{g'(0)/g\in G\}$ and consider the homomorphism $\phi:G\rightarrow G'$ defined by $\phi(g)=g'(0)$. We claim that $\phi$ is injective. In fact, assume that $\phi(g)=1$ and suppose by contradiction that $g\neq id$. In this case $g(z)=z+az^{r+1}+\ldots$, where $a\neq 0$. According to \cite{Loray}, the pseudo-orbits of
this transformation accumulate at $0\in(\sum,0)$, contradicting the fact that the leaves of $\tilde{\mathcal{F}}$ are closed and so the assertion is proved. Now, it suffices to prove that any element $g\in G$ has finite order (cf. \cite{mattei}). In fact, $\phi(g)=g'(0)$ is a root of unity thus $g$ has finite order because $\phi$ is injective. Hence, all transformations of $G$ have finite order and $G$ is linearizable.
\par This implies that there is a coordinate system $w$ on $(\sum,0)$ such that $G=\langle w\rightarrow\lambda w\rangle$, where $\lambda$ is a $d^{th}$-primitive root of unity (cf. \cite{mattei}). In particular, $\psi(w)=w^{d}$ is a first integral of $G$, that is $\psi\circ g=\psi$ for any $g\in G$.
\par Let $\Gamma$ be the union of the separatrices of $\mathcal{L}_{\mathbb{C}}$ through $0\in\mathbb{C}^{6}$ and $\tilde{\Gamma}$ be its strict transform under $\pi$. The first integral $\psi$ can be extended to a first integral $\varphi:\tilde{M}_{\mathbb{C}}\backslash\tilde{\Gamma}\rightarrow\mathbb{C}$ by setting $$\varphi(q)=\psi(\tilde{L}_{q}\cap \sum),$$ where $\tilde{L}_{p}$ denotes the leaf of $\tilde{\mathcal{F}}$ through $q$. Since $\psi$ is bounded (in a compact neighborhood of $0\in\sum$), so is $\varphi$. It follows from Riemann extension theorem that $\varphi$ can be extended holomorphically to $\tilde{\Gamma}$ with $\varphi(\tilde{\Gamma})=0$. This provides the first integral of $\mathcal{L}_{\mathbb{C}}$.
\end{proof}
\par The rest of the proof is devoted to prove that we are indeed in the conditions of Lemma \ref{integral-theorem}. It is follows from Lemma \ref{closed-lemma} that the leaves of $\mathcal{L}_{\mathbb{C}}$ are closed. Therefore, we need to prove that each generator of the holonomy group $G$ of $\tilde{\mathcal{F}}$ with respect to $S$ has finite order.

\par Consider for instance the chart $(U_{1},(x,t,s,z,u,v))$ of $\tilde{\mathbb{C}}^{6}$ where $$\pi(x,t,s,z,u,v)=(x,tu,su,z,u,vu)=(x,y_{1},y_2,z,w_1,w_2).$$
 We have $$\tilde{M}_{\mathbb{C}}\cap U_1=\{(x,t,s,z,u,v)\in U_1| 1+t^2+s^2+v^2+uH_1(x,t,s,z,u,v)=0\},$$
  where $H_{1}=H(x,ut,us,z,u,uv)/u^{3}$ and this fact imply that $$E\cap\tilde{M}_{\mathbb{C}}\cap U_1=\{(x,t,s,z,u,v)\in U_1| 1+t^2+s^2+v^2=u=0\}.$$
It is not difficult to see that these complex subvarieties are smooth.   Now, let us describe the foliation $\tilde{\mathcal{F}}$ on $U_{1}$. In fact, note that the foliation $\mathcal{L}_{\mathbb{C}}$ is defined by $\alpha|_{M^{*}_{\mathbb{C}}}=0$, where
$$\alpha=\frac{1}{2}\frac{\partial{P}}{\partial{x}}dx+\frac{1}{2}\frac{\partial{P}}{\partial{y_1}}dy_1+\frac{1}{2}\frac{\partial{P}}{\partial{y_2}}dy_2+
\frac{\partial{H}_{\mathbb{C}}}{\partial{x}}dx+\sum^{2}_{j=1}\frac{\partial{H}_{\mathbb{C}}}{\partial{y_{j}}}dy_{j}.$$
It follows that $\alpha=y_{1}dy_1+y_2dy_2+\frac{\partial{H}_{\mathbb{C}}}{\partial{x}}dx+\sum^{2}_{j=1}\frac{\partial{H}_{\mathbb{C}}}{\partial{y_{j}}}dy_{j}$, then $\tilde{\mathcal{F}}|_{U_{1}}$ is defined by $\tilde{\alpha}|_{\tilde{M}_{\mathbb{C}}\cap U_1}=0$, where
\begin{eqnarray}\label{A_equa}
\tilde{\alpha}=(t^2+s^2)du+utdt+usds+u\tilde{\theta},
\end{eqnarray}
and $$\tilde{\theta}=\frac{\pi^{*}(\frac{\partial{H_{\mathbb{C}}}}{\partial{x}}dx+\sum^{2}_{j=1}\frac{\partial{H_{\mathbb{C}}}}{\partial{y_{j}}}dy_{j})}{u^{2}}.$$
 Therefore, the singular set of $\tilde{\mathcal{F}}|_{U_{1}}$ is given by $$\sing\tilde{\mathcal{F}}|_{U_1}=\{u=t+is=0\}\cup\{u=t-is=0\}.$$
  On the other hand, note that the exceptional divisor $E$ is invariant by $\tilde{\mathcal{F}}$ and the intersection with $\sing\widetilde{\mathcal{F}}$ is
$$\sing\tilde{\mathcal{F}}|_{U_1}\cap E=\{u=t+is=v^2+1=0\}\cup\{u=t-is=v^2+1=0\}.$$ 
In particular, $S:=(E\cap\tilde{M}_{\mathbb{C}}) \backslash \sing\widetilde{\mathcal{L}}_{\mathbb{C}}$ is a leaf of $\widetilde{\mathcal{F}}$. We calculate the generators of the holonomy group $G$ of the leaf $S$. We work in the chart $U_1$, because of the symmetry of the variables in the definition of the variety $\tilde{M}_{\mathbb{C}}$. 
\par Pick $p_{0}=(0,1,0,0,0,0)\in S\cap U_{1}$ and a transversal $\sum=\{(0,1,0,0,\lambda,0)|\lambda\in\mathbb{C}\}$ parameterized by $\lambda$ at $p_0$.  We have that
$$\sing\tilde{\mathcal{F}}|_{U_1}\cap E=\{u=t+is=v^2+1=0\}\cup\{u=t-is=v^2+1=0\}.$$ 
For each $j=1,2$; let $\rho_j$ be a $2^{td}$-primitive root of $-1$. The fundamental group $\pi_1(S,p_0)$ can be written in terms of generators as
$$\pi_1(S,p_0)=\langle\gamma_j,\delta_j\rangle_{1\leq j\leq 2},$$
where for each $j=1,2$; $\gamma_j$ are loops that turn around $\{u=t+is=v-\rho_j=0\}$ and $\delta_j$
are loops that turns around $\{u=t-is=v-\rho_j=0\}$. Therefore, $G=\langle f_{j},g_{j}\rangle_{1\leq j\leq 2}$, where $f_{j}$ and $g_{j}$ correspond to $[\gamma_j]$ and $[\delta_j]$, respectively. We get from 
(\ref{A_equa}) that $f'_{j}(0)=e^{-\pi i}$ and $g'_{j}(0)=e^{-\pi i}$ for all $1\leq j\leq 2$. The proof of the theorem is complete.
\section{Levi-flat hypersurfaces with a complex line as singularity}
\par In this section, we work with the system of coordinates $(z_1,\ldots,z_n)\in\mathbb{C}^{n}$. The canonical local models examples of Levi-flat hypersurfaces $M$ in $\mathbb{C}^{3}$ such that $\sing(M)=L=\{z_1=z_2=0\}$ are $\{\mathcal{R}e(z^{2}_{1}+z^{2}_{2})=0\}$ and $\{z_1\bar{z}_2-\bar{z}_1z_2=0\}$. 
\par Recently, Burns and Gong \cite{burns} classified, up to local biholomorphism, all germs of quadratic Levi-flat hypersurfaces. Namely, up to biholomorphism, there is only five models:
$$\begin{tabular}{|c|l|r|}\hline
Type & Normal form & Singular set \\
\hline
$Q_{0,2k}$ & $\mathcal{R}e(z_{1}^{2}+z_{2}^{2}+\ldots +z^{2}_{k})$ & $\mathbb{C}^{n-k}$\\
\hline
$Q_{1,1}$ & $z^{2}_{1}+2z^{2}_{1}\bar{z}_{1}+z^{2}_{1}$ & empty\\
\hline
$Q^{\lambda}_{1,2}$ & $z^{2}_{1}+2\lambda z^{2}_{1}\bar{z}_{1}+z^{2}_{1}$ & $\mathbb{C}^{n-1}$ \\
\hline
$Q_{2,2}$ & $(z_{1}+\bar{z}_{1})(z_{2}+\bar{z}_{2})$ & $\mathbb{R}^2\times\mathbb{C}^{n-2}$ \\
\hline
$Q_{2,4}$ & $z_1\bar{z}_2-\bar{z}_1z_2$ & $\mathbb{C}^{n-2}$ \\
\hline

\end{tabular}$$
\begin{center}
Table 2. Levi-flat quadrics 
\end{center}
\par We address the problem of provide conditions to characterize singular Levi-flat hypersurfaces with a complex line as singularity. Using the classification due to Burns and Gong \cite{burns}, it is not hard to prove the following proposition.
\begin{proposition}
Suppose that $M$ is a quadratic real-analytic Levi-flat hypersurface in $\mathbb{C}^{n}$, $n\geq 3$ such that 
$\sing(M)=\{z_{1}=z_{2}=\ldots=z_{n-1}=0\}$. Then
\begin{enumerate}
\item If $n=3$, $M$ is biholomorphically equivalent to $Q_{0,2}$ or  $Q_{2,4}$.
\item  If $n\geq 4$,  $M$ is biholomorphically equivalent to $Q_{0,2(n-1)}$.
\end{enumerate}
\end{proposition}
\begin{proof}
To prove part (1), observe that only there are two models of $M$ which admits $\sing(M)=\{z_1=z_2=0\}$ as singularity,  $Q_{0,2}$ or  $Q_{2,4}$. Now to prove part (2), note that if $n\geq 4$, the real hypersurface $\{z_1\bar{z}_2-\bar{z}_1z_2=0\}$ has a complex subvariety of dimension $n-2$ as singularity. It is follows that $M$ is biholomorphically equivalent to $Q_{0,2(n-1)}$. 
\end{proof}
\par In order to obtain a characterization, we define the Segre varieties associated to real-analytic hypersurfaces. Let $M$ be a real-analytic hypersurface defined by $\{F=0\}$. Fix $p\in M$, the Segre variety associated to $M$ at $p$ is the complex variety in $(\mathbb{C}^{n},p)$ defined by
\begin{equation}\label{segre-variety}
Q_{p}:=\{z\in(\mathbb{C}^{n},p):F_{\mathbb{C}}(z,\bar{p})=0\}.
\end{equation}
\par  Now assume that $M$ is Levi-flat and denote by $L_{p}$ the leaf of $\mathcal{L}$ through $p\in M^{*}$. We denote by $Q'_{p}$ the union of all branches of $Q_{p}$ which are contained in $M$. Observe that $Q'_{p}$ could be the empty set when $p\in\sing(M)$. Otherwise, it is a complex variety of pure dimension $n-1$.
\par The following result is classical, we proved it here for completeness.
\begin{proposition}\label{leaf-proposition}
In above situation, $L_{p}$ is an irreducible component of $(Q_{p},p)$ and $Q'_{p}=L_{p}$. 
\end{proposition}
\begin{proof}
Since $p\in M^{*}$, E. Cartan's theorem assures that there exists a holomorphic coordinate system such that near of $p$, $M$ is given by $\{\mathcal{R}e(z_{n})=0\}$ and $p$ is the origin. In this coordinates system the foliation $\mathcal{L}$ is defined by $dz_{n}|_{M^{*}}=0$. In particular, $L_{0}=\{z_{n}=0\}$ and obviously $\{z_{n}=0\}$ is a branch of $Q_{0}$. Furthermore, $L_{0}$ is the unique germ of complex variety of pure dimension $n-1$ at $0$ which is contained in $M$. Hence $Q'_{0}=L_{0}$. 
\end{proof} 

\par Let $p\in\sing(M)$, we say that $p$ is a \textit{Segre degenerate singularity} if $Q_{p}$ has dimension $n$, that is, $Q_{p}=(\mathbb{C}^{n},p)$. Otherwise, we say that $p$ is a \textit{Segre nondegenerate singularity}. 
\par Suppose that $M$ is defined by $\{F=0\}$ in a neighborhood of $p$, observe that $p$ is a degenerate singularity of $M$ if $z\longmapsto F_{\mathbb{C}}(z,\bar{p})$ is identically zero. 
\begin{remark}\label{degenerate-remark}
 If $V$ is a germ of complex variety of dimension $n-1$ contained in $M$ then for $p\in V$, we have $(V,p)\subset (Q_{p},p)$. In particular, if there exists distinct  infinitely many complex varieties of dimension $n-1$ through $p\in M$ then $p$ is a Segre degenerate singularity.
\end{remark} 

\par To continuation, we consider a germ at $0\in\mathbb{C}^{n}$ of a codimension one
singular holomorphic foliation $\mathcal{F}$.
\begin{definition}
We say that $\mathcal{F}$ and $M$ are tangent, if the leaves of the Levi foliation $\mathcal{L}$ on $M$ are also
leaves of $\mathcal{F}$.
\end{definition}
\begin{definition}
A meromorphic (holomorphic) function $h$ is called a meromorphic (holomorphic) first integral for $\mathcal{F}$ if its indeterminacy (zeros)
set is contained in $\sing(\mathcal{F})$ and its level hypersurfaces contain the leaves of $\mathcal{F}$.
\end{definition}
\par Recently, Cerveau and Lins Neto proved the following result.
\begin{theorem}[Cerveau-Lins Neto \cite{alcides}]\label{lins-cerveau}
 Let $\mathcal{F}$ be a germ at $0\in\mathbb{C}^{n}$, $n\geq{3}$, of holomorphic codimension
one foliation tangent to a germ of an irreducible real
analytic hypersurface $M$. Then $\mathcal{F}$ has a non-constant meromorphic first integral.

\end{theorem}
\par In our context, we prove the following result.
\begin{theorem}\label{line_complex}
Let $M$ be a germ at $0\in\mathbb{C}^{n}$, $n\geq 3$ of an irreducible real-analytic Levi-flat hypersurfaces such that $\sing(M)=L:=\{z_{1}=z_{2}=\ldots=z_{n-1}=0\}$. Suppose that:
\begin{enumerate}
\item Every point in $\sing(M)$ is a Segre nondegenerate singularity.
\item The Levi-foliation $\mathcal{L}$ on $M^{*}$ extends to a holomorphic foliation $\mathcal{F}$ in some neighborhood of $M$.
\end{enumerate}
Then there exists $f\in\mathcal{O}_{n}$ and a real-analytic curve $\gamma\subset\mathbb{C}$ 
 such that $M=f^{-1}(\gamma)$. 
\end{theorem}
\begin{proof}
Since the Levi-foliation $\mathcal{L}$ on $M^{*}$ extends to a holomorphic foliation $\mathcal{F}$, we can apply directly Theorem \ref{lins-cerveau}, this means that $\mathcal{F}$ has a non-constant meromorphic first integral $f=g/h$, where $g$ and $h$ are relatively prime. We asserts that $f$ is holomorphic. In fact, if $f$ is purely meromorphic, we have that for all $\zeta\in\mathbb{C}$, the complex hypersurfaces $V_{\zeta}=\{g(z)-\zeta h(z)=0 \}$ contains leaves of $\mathcal{F}$. In particular, $M$ contains an infinitely many of hypersurfaces $V_{\zeta}$, because $M$ is closed and $\mathcal{F}$ is tangent to $M$. Set
$\Lambda:=\{\zeta\in\mathbb{C}: V_{\zeta}\subset M\}$.
Note also that the foliation $\mathcal{F}$ is singular at $L$, so that $\mathcal{I}_{f}:=\{h=g=0\}$ the indeterminacy set of $f$ intersect $L$.  Therefore, we have a point $q$ at $\mathcal{I}_{f}\cap L$ 
which  would be a Segre degenerate singularity, because $q\in V_{\zeta}$, for all $\zeta\in\Lambda$. It is a contradiction and the assertion is proved.
\par The foliation $\mathcal{F}$ is defined by $df=0$, $f\in\mathcal{O}_n$ and is tangent to $M$. Without loss generality, we can assume that $f$ is an irreducible germ in $\mathcal{O}_n$. According to a remark of Brunella  \cite[pg. 8]{meromorphic}, there exists a real-analytic curve $\gamma\subset\mathbb{C}$ through the origin such that $M=f^{-1}(\gamma)$.
\end{proof}
\begin{remark}
In \cite{singularlebl}, J. Lebl gave conditions for the Levi-foliation on $M^*$ does extended to a holomorphic foliation. 
One could be considered these hypothesis and establish a theorem more refined.
Note also that if $\sing(M)$ is a germ of smooth complex curve, it is possible adapted the proof of Theorem \ref{line_complex}. In general, 
the holomorphic extension problem for the Levi-foliation of a Levi-flat real-analytic hypersurface remains open and is of independent interest, for more details see \cite{generic}.
\end{remark}

\vskip 0.1 in

\noindent{\it\bf Acknowledgments.--}  
 This work was partially supported by PRPq - Universidade Federal de Minas Gerias UFMG 2013 and FAPEMIG APQ-00371-13. I would like to thank Maur\'icio Corr\^ea JR for his comments and suggestions, and the referee for pointing out corrections.

\end{document}